\documentclass[11pt,leqno,twoside]{amsart}
\usepackage{amssymb,amsmath,amsthm,soul,color}
\usepackage{t1enc}
\usepackage[cp1250]{inputenc}
\usepackage{a4,indentfirst,latexsym}
\usepackage{graphics}
\usepackage{mathrsfs}
\usepackage{cite,enumitem,graphicx}
\usepackage[colorlinks=true,urlcolor=blue,
citecolor=red,linkcolor=blue,linktocpage,pdfpagelabels,
bookmarksnumbered,bookmarksopen]{hyperref}
\usepackage[english]{babel}
\usepackage[left=2.61cm,right=2.61cm,top=2.72cm,bottom=2.72cm]{geometry}
\usepackage[colorinlistoftodos]{todonotes}
\usepackage{pifont}
\usepackage{mathrsfs}
\input xy
\xyoption{all}
\usepackage{textcomp} 

\newtheorem{Th}{Theorem}[section]

\newtheorem{Lem}[Th]{Lemma}

\newtheorem{Rem}[Th]{Remark}

\newenvironment{altproof}[1]
{\noindent
{\em Proof of {#1}}.}
{\nopagebreak\mbox{}\hfill $\Box$\par\addvspace{0.5cm}}

\makeatletter
    
    \newcommand{\Rmnum}[1]{\expandafter\@slowromancap\romannumeral #1@}

   \newcommand{\vp}{\varphi}
   
   \newcommand{\eps}{\varepsilon}

   \def\deg{\mathrm{deg}}

   \def\Z{\mathbb{Z}}

   \def\N{\mathbb{N}}
   \def\R{\mathbb{R}}

   \def\P{\mathcal P} 

   \def\J{\mathcal{J}}

      \def\essup{\mathop{\mathrm{ess\; sup}\,}}

\newcommand{\cC}{{\mathcal C}}

\newcommand{\cN}{{\mathcal N}}

\newcommand{\cP}{{\mathcal P}}

\newcommand{\cT}{{\mathcal T}}

\newcommand{\Ga}{\Gamma}

\newcommand{\weakto}{\rightharpoonup}

\newcommand{\tX}{\widetilde{X}}
\newcommand{\tu}{\widetilde{u}}
\newcommand{\tv}{\widetilde{v}}

\newcommand{\cTto}{\stackrel{\cT}{\longrightarrow}}

\numberwithin{equation}{section}

\begin{document}
\thispagestyle{empty}

\title{Ground states of a system of nonlinear Schr\"odinger equations with periodic potentials}

\author{Jaros\l aw Mederski}
\date{}

\maketitle

\pagestyle{myheadings} \markboth{{\small{ J. Mederski}}}{\small{{Ground states of a system of NLS equations}}}

\begin{abstract} We are concerned with a system of coupled Schr\"odinger equations
$$-\Delta u_i + V_i(x)u_i = \partial_{u_i}F(x,u)\hbox{ on }\R^N,\,i=1,2,...,K,$$
where $F$ and $V_i$ are periodic in $x$ and $0\notin \sigma(-\Delta+V_i)$ for $i=1,2,...,K$, where $\sigma(-\Delta+V_i)$ stands for the spectrum of the Schr\"odinger operator $-\Delta+V_i$.
We impose general assumptions on the nonlinearity $F$ with the subcritical growth and we find a ground state solution being a minimizer of the energy functional associated with the system on a Nehari-Pankov manifold. Our approach is based on a new linking-type result involving the Nehari-Pankov manifold.
\end{abstract}

\vspace{0.2cm}
{\bf MSC 2010:} Primary: 35Q60; Secondary: 35J20, 35Q55, 58E05, 35J47

{\bf Keywords:} photonic crystals, gap soliton, ground state, variational methods, strongly indefinite functional, Nehari-Pankov manifold, Bose-Einstein condensates, Schr\"odinger \-system. 

\section*{Introduction}
\setcounter{section}{1}

We consider the following system of nonlinear Schr\"odinger equations of gradient type
\begin{equation}
\label{eq}
\left\{
\begin{array}{ll}
    -\Delta u_1 + V_1(x)u_1 = f_1(x,u)
    &
     \hbox{ in } \R^N,\\
    -\Delta u_2 + V_2(x)u_2 = f_2(x,u)
    &
     \hbox{ in } \R^N,\\
    ...\\
    -\Delta u_K + V_K(x)u_K = f_K(x,u)
    &
     \hbox{ in } \R^N,\\
\end{array}
\right.
\end{equation}
with $u=(u_1,u_2,...,u_K):\R^N\to\R^K$, which arises in different areas of mathematical physics. In particular, if $V_i$ and $f_i$ are periodic in $x$, then there is a wide range of applications in photonic crystals admitting a spatially periodic structure \cite{Pankov,Kuchment,Akozbek}. Then system \eqref{eq} describes the propagation of gap solitons which are special nontrivial solitary wave solutions $\Phi_j(x,t)=u_j(x)e^{-i\omega_j t}$ of a system of time-dependent Schr\"odinger equations of the form
\begin{equation}\label{eq:SchrodTime}
i\frac{\partial \Phi_j}{\partial t}=
-\Delta \Phi_j+(V_j(x)+\omega_j)\Phi_j-g_j(x,\Phi)\quad\hbox{for }j=1,2,...,K,
\end{equation}
where $g_j$ are responsible for nonlinear polarization in a photonic crystal\cite{NonlinearPhotonicCrystals,Akozbek}.

Another field of applications is condensed matter
physics, where \eqref{eq} comes from the system of coupled
Gross-Pitaevski equations \eqref{eq:SchrodTime} with nonlinearities of the form
$$g_j(x,\Phi)=\Big(\sum_{k=1}^K\beta_{j,k}|\Phi_k|^2\Big)\Phi_j\quad\hbox{for }j=1,2,...,K$$ 
and $V_j$ are the external electric potentials which can be periodic \cite{Malomed,Pankov}. Here $\Phi_j(x,t)=u_j(x)e^{-i\omega_j t}$ stand for Bose-Einstein condensates in $K$ different hyperfine spin states.

A general class of autonomous systems of Schr\"odinger equations has been studied by Brezis and Lieb in \cite{BrezisLiebCMP84} and using a constrained minimization method they have shown the existence of a {\em least energy solution}, i.e. a nontrivial solution with the minimal energy. Their method using rescaling arguments does not apply in our case since \eqref{eq} is non-autonomous and $V_i$, $f_i$ are periodic in $x$. There is also an extensive literature devoted to particular power-like nonlinearities in \eqref{eq}, where potentials $V_i$ are positive and constant, asymptotically constant or sufficiently regular (see e.g. \cite{PengWangARMA2013,LinWeiCMP2005,MaiaJDE2006,Sirakov,TerraciniVerzini,WeiWethARMA2008,MontefuscoPellacciSquassinaJEMS2008,AmbrosettiColorado,IkomaTanaka,BartschDancerWang,LinWeiJDE2006,ChenZouCalPDE2013} and references therein). 

Our aim is to study ground state solutions to \eqref{eq} with a general class of nonlinearities. The energy functional $\J$ associated with \eqref{eq} stated below is strongly indefinite and has no longer mountain pass geometry. We provide a collection of assumptions inspired by recent works of Szulkin and Weth \cite{SzulkinWeth} and Liu \cite{Liu} in case of one equation, i.e. $K=1$ (see also Coti Zelati and Rabinowitz \cite{CotiZelati}, Alama and Li \cite{AlamaLi}).

Throughout the paper except the last Section \ref{sect:DecayOfSolutions}, we assume:
\begin{itemize}
 \item[(V)] For $i=1,2,...,K$, $V_i\in L^{\infty}(\R^N)$ is $\Z^N$-periodic, i.e. $V_i(x+z)=V_i(x)$ for $x\in\R^N$, $z\in\Z^N$,   and $0\notin \bigcup_{i=1}^K\sigma(-\Delta+V_i)$, where $\sigma(-\Delta+V_i)$ denotes the spectrum of $-\Delta+V_i$.
\item[(F1)] $f_i:\R^N\times\R^K\to \R$ is measurable, $\Z^N$-periodic in $x\in\R^N$ and continuous in $u\in\R^K$ for a.e. $x\in\R^N$. Moreover $f=(f_1,f_2,...,f_K)=\partial_u F$, where $F:\R^N\times\R^K\to\R$ is differentiable with respect to the second variable $u\in\R^K$ and $F(x,0)=0$ for a.e. $x\in \R^N$.

\item[(F2)] There are $a>0$ and $2<p<2^*=\frac{2N}{(N-2)_+}$ such that
$$|f(x,u)|\leq a(1+|u|^{p-1})\quad\hbox{for all }u \in\R^K\hbox{ and a.e. }x\in\R^N.$$

\item[(F3)] $f(x,u)=o(u)$ uniformly with respect to $x$ as $|u|\to0$.
\end{itemize}

The energy functional
$\J:H^1(\R^N)^K\to\R$ given by
\begin{eqnarray*}
\J(u)&=&\frac12\sum_{i=1}^K\int_{\R^N}|\nabla u_i|^2+V_i(x)|u_i|^2\,dx- \int_{\R^N} F(x,u)\, dx
\end{eqnarray*}
is of $\cC^1$-class and its critical points correspond to solutions of \eqref{eq}. In view of (V) spectral theory asserts that $\sigma(-\Delta+V_i)$ is purely continuous, bounded from below and consists of closed disjoint intervals \cite{ReedSimon}. Moreover
there is an orthogonal decomposition of $X:=H^1(\R^N)^K=X^+\oplus\tX$, such that the second variation $\J''(0)[u,u]$ is positive definite on $X^+$ and negative definite on $\tX$ (see Section \ref{sect:VariationalSetting} for details). If $0$ lies in a finite spectral gap, i.e. $0\notin\sigma(-\Delta+V_i)$ and a part of the spectrum $\sigma(-\Delta+V_i)$ lies below $0$ for some $i=1,2,...,K$, then $X^+$ and $\tX$ are infinite dimensional and the problem is strongly indefinite.

Our goal is to find a {\em ground state} of $\J$, i.e. a critical point  being a minimizer of $\J$ on the Nehari-Pankov {\em manifold} defined as follows
$$\cN=\{u\in X\setminus \tX|\, \J'(u)(u)=0\hbox{ and } \J'(u)(v)=0\hbox{ for any }v\in \tX\}.$$
Since $\cN$ contains all nontrivial critical points, then a ground state is a least energy solution.  In order to deal with the geometry of $\J$ and to set up $\cN$ we need the following conditions:

\begin{itemize}
\item[(F4)] $f(x,u)u\geq 2 F(x,u)\geq0$ for all $u\in\R^K$ and a.e. $x\in\R^N$.

\item[(F5)] $F(x,u)/|u|^2\to\infty$ uniformly in $x$ as $|u|\to\infty$.

\item[(F6)] If $f(x,u)v=f(x,v)u> 0$, then 
$\ \displaystyle F(x,u) - F(x,v)
 \le \frac{(f(x,u)u)^2-(f(x,u)v)^2}{2f(x,u)u}.$
\end{itemize}

Observe that if $F(x,u)=\Gamma(x)W(|Mu|^2)$, where  $\Ga\in L^\infty(\R^3)$ is $\Z^3$-periodic, positive and bounded away from $0$,
$W\in \cC^1(\R,\R)$, $W(0)=W'(0)=0$, $W'(t)$ is nondecreasing on $(0,+\infty)$ and $M\in GL(K)$ is an invertible $K\times K$ matrix, then assumptions (F1), (F4) and (F6) are satisfied. Indeed, (F1), (F2) are clear, and the assumption $f(x,u)v=f(x,v)u\neq 0$ implies that $F(x,u)=F(x,v)$. The remaining assumptions (F2), (F3) and (F5) can be verified by suitable growth conditions imposed on $W$. For instance, in Kerr photonic crystals one has $W(t)=t^2$. More examples will be provided in Remark \ref{remF6}.

We state our main result.
\begin{Th}\label{ThMain1}
Suppose that (V), (F1)-(F6) are satisfied. Then \eqref{eq} has a ground state, i.e. there is a nontrivial critical point $u$ of $\J$ such that $\J(u)=\inf_{\cN}\J$. Moreover (V), (F1)-(F3) imply that any solution $u$ of \eqref{eq} is continuous and there exist $\alpha, C>0$ such that
$$|u(x)|\leq C \exp(-\alpha |x|)\hbox{ for any }x\in\R^N.$$
\end{Th}

Observe that condition (F6) can be regarded as a vector version of the following weak monotonicity condition for $K=1$ 
\begin{equation}\label{eqwMC}
u\mapsto f(x,u)/|u|\hbox{ is nondecreasing on }(-\infty,0)\hbox{ and on }
(0,\infty),
\end{equation}
namely \eqref{eqwMC} implies (F6) in dimension $K=1$.
In order to apply a Nehari-Pankov manifold technique, Szulkin and Weth \cite{SzulkinWeth} required the strong  monotonicity condition:
\begin{equation}\label{eqMC}
u\mapsto f(x,u)/|u|\hbox{ is strictly increasing on }(-\infty,0)\hbox{ and on }
(0,\infty),
\end{equation}
and they obtained the existence of a ground state of the Schr\"odinger equation, i.e. \eqref{eq} with $K=1$, under slightly more restrictive assumptions than (V) and (F1), and assuming additionally (F2), (F3), (F5) and \eqref{eqMC}. Observe that if $K=1$ then (F4) follows from (F3) and \eqref{eqwMC}, so that assumptions (F1)-(F6) are more general than those in \cite{SzulkinWeth}. Hence Theorem \ref{ThMain1}, in case $K=1$, contains \cite{SzulkinWeth}[Theorem 1.1] as well as \cite{Liu}[Theorem 1.1].

If we replace \eqref{eqMC} by the weak monotonicity condition \eqref{eqwMC}, then the methods of \cite{SzulkinWeth} fail. 
Moreover the recently generalized Nehari-Pankov manifold technique due to Bartsch and the author \cite{BartschMederski1}[Section 4] is still insufficient to treat our system. Indeed, if $K=1$, then the weak monotonicity condition \eqref{eqwMC} does not imply that $\J$ has  the unique global maximum in $\R^+u\oplus\tX$ for any $u\in X^+\setminus\{0\}$, hence $\cN$ may not be homeomorphic with a sphere in $X^+$. Therefore we are not able to minimize $\J$ on a sphere to get critical points. Moreover we do not assume that $f_i$ are differentiable, thus $\cN$ need not to be of class $\cC^1$ and the standard minimizing methods on $\cN$ do not apply. Hence, $\cN$ is just a subset of $X$, however we will use the term {\em manifold}.
Our approach is based on a deformation argument and we obtain a new linking-type result involving the Nehari-Pankov manifold in Theorem \ref{ThAbstract}.

In the paper we provide also two different minimax characterizations of the ground state level, i.e.
\begin{equation}\label{eq:minmax}
\inf_\cN\J=\inf_{u\in X^+\setminus\{0\}}\sup_{t\geq 0,\,\tu\in\tX}\J(tu^++\tu)
=\inf_{u\in X\setminus \tX}\inf_{h\in\Gamma(u)} \sup_{u'\in M(u)} \J(h(u',1)), 
\end{equation}
where $\Gamma(u)$ consists of admissible homotopies on $M(u)\times [0,1]$ defined in Section \ref{sect:CriticalPoitTheory}. Note that \eqref{eq:minmax} is important from the numerical point of view and gives rise to compute the ground state level and ground state solutions (see Li and Zhou \cite{LiZhouNumer2001}).

The paper is organized as follows. In the next section we present a critical point theory for a class of functionals like $\J$ and we provide a linking-type result involving the Nehari-Pankov manifold. In Section \ref{sect:VariationalSetting}
we formulate our problem in the variational setting and prove Theorem \ref{ThMain1}. Finally we study the continuity and the exponential decay of solutions to \eqref{eq} in a general setting and, due to possible applications to a wider range of systems, we place it in a self-contained Section \ref{sect:DecayOfSolutions}.

\section{Critical point theory - ground states via linking}
\label{sect:CriticalPoitTheory}

Let $X = X^+\oplus \tX$ be a Hilbert space with $X^+$ orthogonal to $\tX$, and $\tX$ is separable.  
For $u \in X$ we denote by $u^+ \in X^+$ and $\tu  \in \tX$ the corresponding summands so that $u = u^++\tu$.
In addition to the norm topology $\|\cdot\|$ we need the topology $\cT$ on $X$ which is induced by the norm
$$\|u\|_\cT:=\max\Big\{\|u^+\|, \sum_{k=1}^\infty\frac{1}{2^{k+1}}|\langle \tu,e_k\rangle|\Big\},$$
where $(e_k)$ stands for a total orthonormal sequence in $\tX$. Recall that \cite{KryszSzulkin,Willem,BartschDing,LiSzulkin,WillemZou} 
$$\|\tu\|\leq\|u\|_\cT\leq \|u\|\hbox{ for }u\in X$$
and on bounded subsets of $X$ the topology $\cT$ coincides with
the
product of the norm topology in $X^+$ and the weak topology in $\tX$. The convergence of a sequence in $\cT$ topology will be denoted by $u_n\cTto u$.

We consider a functional $\J\in\cC^1(X,\R)$ 
such that the following conditions hold:
\begin{itemize}
\item[(A1)] $\J$ is $\cT$-upper semicontinuous, i.e. $\J_t:=\J^{-1}([t,\infty))$ is $\cT$-closed for any $t\in\R$.
\item[(A2)] $\J'$ is $\cT$-to-weak$^*$ continuous in $\J_0$, i.e. $\J'(u_n)\weakto\J'(u_0)$ if $u_n\cTto u_0$ and $(u_n)_{n\geq 0}\subset \J_0$.
\end{itemize}

Let $\P\subset X\setminus\tX$ and $\cP\neq\emptyset$. The linking geometry of $\J$ is described by the following assumptions.

\begin{itemize}
\item[(A3)] There exists $r>0$ such that $d:=\inf\limits_{u\in X^+,\|u\|=r} \J(u)>0$.
\item[(A4)] For every $u\in \cP$ there exists $R(u)>r$ such that
$$\sup_{\partial M(u)}\J\leq \J(0)=0,$$
where $M(u):=\{t u+v \in X |\;t\geq0,\; v \in \tX\hbox{ and }
\|tu+v\|\leq R(u)\}$.
\end{itemize}

In the usual linking geometry one assumes (A3) and (A4) when $\cP$ is a singleton \cite{BartschDing,KryszSzulkin,WillemZou,LiSzulkin,BenciRabinowitz}. In our case $\cP$ joins the linking geometry with the following set
\begin{equation}\label{DefOfN}
\cN := \{u\in \cP|\; \J'(u)(u)=0\hbox{ and }\J'(u)(v)=0\hbox{ for any }v\in \tX\}.
\end{equation}
If $\cP=X\setminus\tX$, then $\cN$ has been introduced by Pankov in \cite{Pankov} (see also \cite{SzulkinWeth,BartschMederski1}).
The following condition is considered as well.
\begin{itemize}
\item[(A5)] If $u\in \cN$ then $\J(u)\geq \J(tu+v)$ for $t\geq 0$ and $v\in \tX$.
\end{itemize}

For any $A\subset X$, $I\subset [0,+\infty)$ such that $0\in I$, and $h:A\times I\to X$ we collect the following assumptions.
\begin{itemize}
\item[(h1)] $h$ is $\cT$ -continuous (with respect to norm $\|\cdot\|_\cT$);
\item[(h2)] $h(u,0)=u$ for all $u\in A$;
\item[(h3)] $\J(u)\geq \J(h(u,t))$ for all $(u,t)\in A\times I$;
\item[(h4)] each $(u,t)\in A\times I$ has an open neighborhood $W$ in the product topology of $(X,\cT)$ and $I$ such that the set $\{v-h(v,s):(v,s)\in W\cap(A\times I)\}$ is contained in a finite-dimensional subspace of $X$. 
\end{itemize}

\begin{Th}\label{ThAbstract}
Suppose that $\J\in\cC^1(X,\R)$ satisfies (A1)-(A4). Then  there exists a Cerami sequence $(u_n)$ at level $c$, i.e.
$\J(u_n)\to c$ and $(1+\|u_n\|)\J'(u_n)\to 0$, where
\begin{eqnarray*}
c&:=&\inf_{u\in \cP}\inf_{h\in\Gamma(u)} \sup_{u'\in M(u)} \J(h(u',1))\geq d>0,\\
\Gamma(u)&:=&\{h\in \cC(M(u)\times [0,1])|\;h \hbox{ satisfies } (h1)-(h4)\}.
\end{eqnarray*}
Suppose that in addition (A5) holds. Then 
$c\leq \inf_\cN\J$, and if $c\geq\J(u)$ for some critical point $u\in\cP$, then
$$c=\inf_\cN\J=\J(u).$$
\end{Th}

\begin{proof} Proof follows from the following Steps 1-4. In Steps 1-3 we employ similar arguments to those given in \cite{KryszSzulkin,LiSzulkin,WillemZou,Willem,BartschDing}. Therefore we omit some details and provide the appropriate references. 
For any $s<t$ we denote  $\J_s^t:=\J^{-1}([s,t])$ and $\J^t:=\J^{-1}((-\infty,t])$.\\
{\em Step 1.} We show that $c\geq d$.\\
Take any $u\in \cP$ and $h\in\Gamma(u)$. Observe that the map $H:M(u)\times[0,1]\to \R u^+\oplus \tX$ given by the formula
$$H(v,t):=(\|h(v,t)^+\|-r)\frac{u^+}{\|u^+\|}+\widetilde{h}(v,t)$$
is admissible, i.e. (h1) and (h4) hold. Moreover $H(v,t)=0$ if and only if $h(v,t)\in X^+$ and $\|h(v,t)\|=r$. Since $0\notin H(\partial M(u)\times [0,1])$ then by the homotopy invariance and the existence property of the degree provided in \cite{KryszSzulkin} we get
$$\deg (H(\cdot,1),M(u))=\deg (H(\cdot,0),M(u))=1\neq 0.$$
Therefore $H(v,1)=0$ for some $v\in M(u)$, thus $h(v,t)\in X^+$ and $\|h(v,t)\|=r$. Observe that
$$\sup_{u'\in M(u)} \J(h(u',1))\geq \J(h(v,1))\geq d.$$
Hence $c\geq d$.\\
\noindent {\em Step 2.} Suppose that there is no Cerami sequence at level $c$, i.e. there exists $\eps>0$ such that $(1+\|u\|)\|\J'(u)\|\geq \eps$ for any $u\in \J_{c-\eps}^{c+\eps}$. Then there is a continuous and $\cT$-continuous $\eta:\J^{c+\eps}\times [0,\infty)\to X$ such that $\eta|_{\J^{c+\eps}\times[0,2\eps]}$ satisfies (h1)-(h4) and 
\begin{equation}\label{eq:Link3}
\eta(\J^{c+\eps}\times\{2\eps\})\subset \J^{c-\eps}. 
\end{equation}
Indeed, as in proof in \cite{LiSzulkin}[Theorem 2.1] and \cite{Willem}[Lemma 6.7] there is a $\cT$-locally-Lipschitz and locally Lipschitz vector field $V:N\to X$ such that $N$ is a $\cT$-open neighborhood of $\J^{c+\eps}$ and any $u\in N$ has $\cT$-open neighborhood $N_u$ such that $V(N_u)$ is contained in a finite-dimensional subspace of $X$. Moreover for some constant $C>0$
\begin{eqnarray}\label{eq:Link1}
\|V(u)\|\leq C(1+\|u\|)&& \hbox{ for }u\in N,\nonumber\\
\langle \J'(u),V(u)\rangle \geq 0 && \hbox{ for }u\in \J^{c+\eps},\\
\langle \J'(u),V(u)\rangle >1 && \hbox{ for }u\in \J_{c-\eps}^{c+\eps}.
\label{eq:Link2}
\end{eqnarray}
Hence there is a unique solution $\eta:\J^{c+\eps}\times [0,\infty)\to X$ of the initial value problem
$$\partial_t\eta(u,t)=-V(\eta(u,t)),\quad\eta(u,0)=u\in\J^{c+\eps}.$$ Observe that \eqref{eq:Link1} implies (h3) and \eqref{eq:Link2} implies that
$\eta(\J^{c+\eps}\times \{2\eps\})\subset \J^{c-\eps}$. Moreover $\eta|_{\J^{c+\eps}\times[0,2\eps]}$ satisfies (h1) and (h4) as in \cite{LiSzulkin,WillemZou,KryszSzulkin}.\\
\noindent {\em Step 3.} Take any $u\in \P$ and $h\in\Gamma(u)$ such that $\sup_{u'\in M(u)}\J(h(u',1))<c+\eps$. Observe that $g:M(u)\times [0,1]\to X$ given by
\begin{equation*}
g(u',t):=\left\{
  \begin{array}{lcl}
    h(u',2t)  &
    t\in [0,1/2],\\
    \eta(h(u',1),2\eps(2t-1))
    &
    t\in [1/2,1]
  \end{array}
\right.
\end{equation*}
satisfies (h1)-(h4) and $g\in\Gamma(u)$. In view of \eqref{eq:Link3} we get  $$\sup_{u'\in M(u)}\J(g(u',1))\leq c-\eps$$
which contradicts the definition of $c$.\\
\noindent {\em Step 4.} Suppose that (A5) holds.
If $\cN=\emptyset$, then $\inf_\cN\J=\infty$. Let $\cN\neq\emptyset$, take any $u\in \cN$ and observe that $h:M(u)\times [0,1]\to X$ such that $h(u',t)=u'$ for $u'\in M(u)$, satisfies (h1)-(h4). From (A5) we get
$$c\leq \J(h(u',t))=\J(u')\leq \J(u).$$
Therefore $c\leq \inf_\cN\J$. Moreover, if $c\geq\J(u)$ and $u\in\cP$ is a critical point, then $u\in\cN$ and $c=\inf_\cN\J=\J(u)$.
\end{proof}

Observe that if $\tX=\{0\}$, then (A1) and (A2) are trivially satisfied for $\J$ of class $\cC^1$. Moreover (A3) and (A4) coincide with the classical assumptions of the mountain pass geometry provided that $\cP$ is a singleton \cite{Willem,Struwe}. If $\cP$ contains the classical Nehari manifold  
\[
\cN := \{u\in X\setminus \{0\}|\; \J'(u)(u)=0\},
\]
then Theorem \ref{ThAbstract} becomes a variant of the Mountain Pass Theorem and the mountain pass level $c$ coincides with the ground state level provided that $c$ is a critical value (cf. \cite{Willem}[Theorem 4.2]). As opposed to the usual Nehari (or Nehari-Pankov) manifold approaches (e.g. \cite{Willem,SzulkinWeth,BartschMederski1,Pankov}), in (A5) we do not require that $\cN$ consists of the unique maximum points of $\J$ on $\R^+u\oplus\tX$ for $u\in X^+\setminus\{0\}$. Therefore we are able to consider a wider range of nonlinearities.

\section{Variational setting}
\label{sect:VariationalSetting}

Since $0\notin\sigma(-\Delta+V_i)$, then the spectral theory asserts that we may find continuous projections $P_i^+$ and $P_i^-$ of $H^1(\R^N)$ onto $X_i^+$ and $X_i^-$ respectively such that $H^1(\R^N)=X^+_i\oplus X^-_i$ for $i=1,2...,K$, see \cite{ReedSimon}. Moreover we introduce new inner products in $H^1(\R^N)$ by the following formulas
\begin{eqnarray*}
\langle u,v\rangle_i &:=&\int_{\R^N}\langle \nabla P_i^+(u),\nabla P_i^+(v)\rangle+V_i(x)\langle P_i^+(u),P_i^+(v)\rangle\,dx\\
&&-
\int_{\R^N}\langle \nabla P_i^-(u),\nabla P_i^-(v)\rangle+V_i(x)\langle P_i^-(u),P_i^-(v)\rangle\,dx
\end{eqnarray*}
and norms given by
$$\|u\|_i:=(\langle u,u\rangle_i)^{1/2}$$
for $i=1,2,...,K$.
Let
\begin{eqnarray*}
X^+&:=&X_1^+\times X_2^+\times ...\times X_K^+\\
\tX&:=&X_1^-\times X_2^-\times ...\times X_K^- 
\end{eqnarray*}
and observe that  any $u\in X:=H^1(\R^N)^K$ admits a unique decomposition $u=u^++\tu$, where
$u^+=(P^+_1(u_1),P^+_2(u_2),...,P^+_K(u_K))\in X^+$ and $\tu=(P^-_i(u_i),P^-_2(u_2),...,P^-_K(u_K))\in \tX$. We introduce a new norm in $X$ given by
\begin{equation*}
\|u\|^2= \sum_{i=1}^K(\|P^+_i(u_i)\|^2_i+\|P^-_i(u_i)\|^2_i)=
\sum_{i=1}^K\|u_i\|^2_i.
\end{equation*}
Then
\begin{eqnarray*}
\J(u)&=&\frac12\sum_{i=1}^K\int_{\R^N}|\nabla u_i|^2+V_i(x)|u_i|^2\,dx- \int_{\R^N} F(x,u)\, dx\\
&=& \frac12\sum_{i=1}^K(\|P^+_i(u_i)\|^2_i-\|P^-_i(u_i)\|^2_i)- \int_{\R^N} F(x,u)\, dx
= \frac12\sum_{i=1}^K\|P^+_i(u_i)\|^2_i-I(u)\\
&=& \frac12\|u^+\|^2-I(u),
\end{eqnarray*}
where
\begin{eqnarray*}
I(u)&:=&\frac12\sum_{i=1}^K\|P^-_i(u_i)\|^2_i+ \int_{\R^N} F(x,u)\,dx\\
&=&\frac12\|\tu\|^2+ \int_{\R^N} F(x,u)\,dx.
\end{eqnarray*}
Observe that for every $\eps>0$ there is $C_{\eps}>0$ such that
\begin{equation}\label{eq:estimationOf_f}
|F(x,u)|\leq \eps|u|^2+C_{\eps}|u|^{p}\quad\hbox{for }u\in\R^K,
\end{equation}
and $\J,I\in\cC^1(X,\R)$. Let
\begin{equation}\label{DefOfP}
\cP=\{u\in X\setminus\tX|\,I'(u)(u)> 0\}
\end{equation}
and note that
$$\cN=\{u\in X\setminus \tX|\, \J'(u)(u)=0\hbox{ and } \J'(u)(v)=0\hbox{ for any }v\in \tX\}\subset \P.$$

Suppose that (F5) additionally holds. Then similarly as in \cite{KryszSzulkin,Willem,LiSzulkin,WillemZou} 
we check that conditions (A1)-(A3) are satisfied. Moreover $\J$ has the linking geometry.

\begin{Lem}
Condition (A4) is satisfied provided that (F5) holds.
\end{Lem}
\begin{proof}
Suppose that $u\in\cP$ and there are $t_n>0$ and $v_n\in\tX$ such that $\J(t_nu+v_n)>0$ and $\|t_nu+v_n\|\to\infty$ as $n\to\infty$. 
Let $w_n=\frac{t_nu+v_n}{\|t_nu+v_n\|}$ and we may assume that $w_n\weakto w$ in $X$ and $w_n(x)\to w(x)$ a.e. in $\R^N$ for some $w\in X$. Since 
$$0<\J(t_nu+v_n)\leq \frac12\|t_nu^+\|^2-\frac12\|t_n\tu+v\|^2$$
then
$$\frac{1}{2}< \frac{t_n^2}{\|t_nu+v_n\|^2}\|u^+\|^2=\|w_n^+\|^2\leq \|w_n\|^2=1$$
and we may assume that $w^+\neq 0$. Hence $w\neq 0$ and $|t_nu(x)+v_n(x)|=|w_n(x)|\|t_nu+v_n\|\to\infty$ as $n\to\infty$ and $w(x)\neq 0$. Then by Fatou's lemma
\begin{eqnarray*}
0<\frac{\J(t_nu+v_n)}{\|t_nu+v_n\|^2}\leq  \frac{1}{2}-\int_{\R^N}\frac{F(x,t_nu+v_n)}{|t_nu+v_n|^2}|w_n|^2\,dx\to-\infty \quad\hbox{ as }t\to\infty
\end{eqnarray*}
and we obtain a contradiction.
\end{proof}

\begin{Lem}\label{LemB3check}
If (F4)-(F6) hold, $u\in X$, $v\in\tX$ and $t\geq 0$ then
\begin{equation}\label{eq:LemB3check}
\J(u)\geq \J(tu+v) -\J'(u)\left(\frac{t^2-1}{2}u+tv\right).
\end{equation}
In particular, condition (A5) holds.
\end{Lem}

\begin{proof}
Let $u\in X$, $v\in\tX$ and $t\geq 0$.
Observe that
\begin{eqnarray*}
\J(tu+v)-\J(u)-\J'(u)\left(\frac{t^2-1}{2}u+tv\right)
&=& I'(u)\left(\frac{t^2-1}{2}u+tv\right)
  + I(u) - I(tu+v)\\
&=&-\frac{1}{2}\|v\|^2 +\int_{\R^N}\vp(t,x)\,dx
\end{eqnarray*}
where
$$\vp(t,x):=f(x,u)\Big(\frac{t^2-1}{2}u+tv\Big)
  + F(x,u) - F(x,tu+v).$$
We show that $\vp(t,x)\leq 0$ for $t\geq 0$ and $x\in\R^N$.
Suppose that $u(x)\neq 0$. Then by (F4) we have $\vp(0,x)\leq 0$ and by (F5) we get
$\vp(t,x)\to-\infty$ as $t\to\infty$. Let $t_0\geq 0$ be such that
$$\vp(t_0,x)=\max_{t\geq 0}\vp(t,x).$$
We may assume that $t_0>0$ and thus $\partial_t\vp(t_0,x)=0$. Therefore
$$f(x,u)(t_0u+v)=f(x,t_0u+v)u$$
and if $f(x,u)(t_0u+v)=f(x,t_0u+v)u\leq 0$ then by (F4)
$$\vp(t_0,x)\leq \frac{-t_0^2-1}{2}f(x,u)u
  + F(x,u) - F(x,t_0u+v)\leq 0.$$
Otherwise (F6) implies that
\begin{eqnarray*}
\vp(t_0,x)&\leq &f(x,u)\Big(\frac{t_0^2-1}{2}u+t_0v\Big)
  + \frac{(f(x,u)u)^2-(f(x,u)(t_0u+v))^2}{2f(x,u)u}\\
&=&-\frac{(f(x,u)v)^2}{2f(x,u)u}\leq 0.
\end{eqnarray*}
\end{proof}

\begin{Rem}\label{remF6} 
(a) The inspection of proof of Lemma \ref{LemB3check} shows that conditions (F4)-(F6) imply that
\begin{equation}\label{neqF6}
f(x,u)\Big(\frac{t^2-1}{2}u+tv\Big)
  + F(x,u) - F(x,tu+v)\leq 0\hbox{ for any }u,v\in\R, t\geq 0\hbox{ and a.e }x\in\R^N.
\end{equation}
On the other hand, \eqref{neqF6} implies (F4) and (F6). Indeed, take $v=0$, $t=0$, and next take $u=0$ in \eqref{neqF6}. Then we easily get (F4). Suppose that $f(x,u)\neq 0$. In order to see (F6), note that by replacing $v$ with $-tu+v$ in \eqref{neqF6} and taking $t=\frac{f(x,u)v}{f(x,u)u}>0$ we get
\begin{eqnarray*}
f(x,u)\Big(-\frac{t^2+1}{2}u+tv\Big)
  + F(x,u) - F(x,v)&=& \frac{(f(x,u)v)^2-(f(x,u)u)^2}{2f(x,u)u}+F(x,u) - F(x,v)\\
  &\leq &0.\
\end{eqnarray*} 
In particular, if $f(x,u)v=f(x,v)u>0$, then $f(x,u)\neq 0$ and by (F4), $f(x,u)u> 0$. Thus (F6) is satisfied. Therefore the collection of assumptions (F1)-(F6) is equivalent with the following one: (F1)-(F3), (F5) and \eqref{neqF6}.\\
(b) Observe that if $F,G:\R^N\times \R$ satisfies (F1)-(F6)
with $f=\partial_u F$, $g=\partial_u G$, then $\alpha F+\beta G$ satisfies (F1)-(F6) for any $\alpha,\beta>0$. Indeed, let $F$ and $G$ satisfy (F1)-(F6). It is not clear whether (F6) can be checked directly for $\alpha F+\beta G$. However, in view of (a), \eqref{neqF6} holds for $F$ and $G$, hence for $\alpha F+\beta G$ as well. Therefore  $\alpha F+\beta G$ satisfies (F1)-(F6) for any $\alpha,\beta>0$. In particular,
it is easy to see that
\begin{equation}\label{Ex1}
F(x,u)=\sum_{i=1}^m\frac{1}{p_i}|\Ga_i(x)u|^{p_i}
\end{equation}
satisfies (F1)-(F6), provided that $2<p_1\le p_2\le\dots\le p_m<2^*$, $\Ga_i(x)\in GL(K)$ for a.e.~$x\in\R^N$, and $\Ga_i,\Ga_i^{-1}\in L^\infty(\R^N,\R^{K\times K})$.\\
(c) Observe that, taking $t=1$ in \eqref{neqF6} we see that $F(x,\cdot)$ must be convex for a.e. $x\in \R^N$.
\end{Rem}

\begin{Lem}\label{LemCeramiBound1}
If (F4)-(F6) hold, $(u_n)$ is a Cerami sequence at level $c>0$, then $(u_n)$ is bounded.
\end{Lem}
\begin{proof}
Suppose that $(u_n)$ is a Cerami sequence at level $c>0$ such that $\|u_n\|\to\infty$ as $n\to\infty$. Let $v_n:=\frac{u_n}{\|u_n\|}$. We may assume that $v_n\rightharpoonup v$ in $X$ and $v_n(x)\to v(x)$ a.e. in $\R^N$. Moreover there is a sequence $(y_n)_{n\in\N}\subset\R^N$ such that
\begin{equation}\label{EqLemLions}
\liminf_{n\to\infty}\int_{B(y_n,1)}|v_n^+|^2\,dx >0.
\end{equation}
Otherwise, in view of Lions lemma (see \cite{Willem}[Lemma 1.21]) we get that $v_n^+\to 0$ in $L^t(\R^N)^K$ for $2<t<2^*$. 
By \eqref{eq:estimationOf_f} we obtain
$\int_{\R^N}F(x,sv_n^+)\,dx\to 0$ for any $s>0$. Let us fix $s>0$. Observe that $\J'(u_n)(u_n)\to 0$ and $\J'(u_n)(u_n^-)\to 0$ as $n\to\infty$. Then by \eqref{eq:LemB3check} we get
\begin{equation}\label{EqIneq2}
c=\limsup_{n\to\infty}\J(u_n)\geq \limsup_{n\to\infty} \J(sv_n^+)= \frac{s^2}{2}\limsup_{n\to\infty}\|v_n^+\|^2.
\end{equation}
Note that
$$\frac{1}{2}(\|u_n^+\|^2-\|\tu_n\|^2)\geq \J(u_n)\geq \frac{c}{2}>0$$
for sufficiently large $n$. Hence
\begin{eqnarray*}
2\|u_n^+\|^2&\geq& \|u_n^+\|^2+\|\tu_n\|^2+c\geq\|u_n\|^2
\end{eqnarray*}
and, passing to a subsequence if necessary, $C_2:=\inf_{n\in\N}\|v_n^+\|^2>0$. Then
 by (\ref{EqIneq2})
$$c\geq \frac{s^2}{2}C_2$$
for any $s\geq0$ and the obtained contradiction shows that (\ref{EqLemLions}) holds. We may assume that $(y_n)\subset\Z^N$ and
$$\liminf_{n\to\infty}\int_{B(y_n,r)}|v_n^+|^2\,dx >0$$
for some $r>1$. Since $\J$ and $\mathcal{N}$ are invariant under translations of the form $u\mapsto u(\cdot-k)$, $k\in\Z^N$, then we may assume that $v_n^+\to v^+$ in $L^2_{loc}(\R^N)^K$ and $v^+\neq 0$. Note that if $v(x)\neq 0$ then $|u_n(x)|=|v_n(x)|\|u_n\|\to\infty$
and by (F5)
$$\frac{F(x,u_n(x))}{\|u_n\|^2}=\frac{F(x,u_n(x))}{|u_n(x)|^2}|v_n(x)|^2\to\infty$$
as $n\to\infty$.
Therefore by Fatou's lemma
\begin{eqnarray*}
\frac{\J(u_n)}{\|u_n\|^2}&=&\frac{1}{2}(\|v_n^+\|^2-\|\tv_n\|^2)
-\int_{\R^N}\frac{F(x,u_n(x))}{\|u_n\|^2}\,dx\\
&\to&-\infty.
\end{eqnarray*}
Thus we get a contradiction.
\end{proof}

\begin{altproof}{Theorem \ref{ThMain1}}
Observe that (F4) implies that $\cP=X\setminus\tX$.
In view of Theorem \ref{ThAbstract} there is a Cerami sequence $(u_n)$ at level $c>0$ and by Lemma \ref{LemCeramiBound1}, $(u_n)$ is bounded. Then passing to a subsequence we may assume that 
$u_n\rightharpoonup u$ in $X$ 
and there is a sequence $(y_n)\subset\R^N$ such that
\begin{equation}\label{EqLemLions1_1}
\liminf_{n\to\infty}\int_{B(y_n,1)}|u_n^+|^2\,dx >0.
\end{equation}
Otherwise, in view of Lions lemma, $u_n^+\to 0$ in $L^t(\R^N)^K$ for $2<t<2^*$. By \eqref{eq:estimationOf_f} 
we obtain
$$\|u_n^+\|^2=\J'(u_n)(u_n^+)+\int_{\R^N}f(x,u_n)u_n^+\, dx\to 0$$
as $n\to\infty$. Hence
$$0<c=\lim_{n\to\infty}\J(u_n)\leq\lim_{n\to\infty}\frac{1}{2}\|u_n^+\|^2=0$$
and we get a contradiction. Therefore (\ref{EqLemLions1_1}) holds 
and we may assume that there is a sequence $(y_n)\subset\Z^N$ such that
\begin{equation}\label{EqLemLions2_1}
\liminf_{n\to\infty}\int_{B(y_n,r)}|u_n^+|^2\, dx >0
\end{equation}
for some $r>1$. 
Since $\|u_n(\cdot + y_n)\|=\|u_n\|$, then there is $u\in X$ such that, up to a subsequence, $u_n(\cdot + y_n)\rightharpoonup u$ in $X$, $u_n(x+y_n)\to u(x)$ a.e. on $\R^N$ and $u_n^+(\cdot + y_n)\to u^+$ in $L^2_{loc}(\R^N)^K$.
By (\ref{EqLemLions2_1}) we get $u^+\neq 0$ and then $u\neq 0$.
Since $\J$ and $\mathcal{N}$ are invariant under translations of the form $u\mapsto u(\cdot+y)$, $y\in\Z^N$, then $\J'(u)=0$. Observe that 
by (F4)
$$\frac{1}{2} f(x,u_n(x+y_n))u_n(x+y_n)-
F(x,u_n(x+y_n))\geq 0.$$
Therefore, in view of Fatou's lemma
$$c=\lim_{n\to\infty}\J(u_n(\cdot+y_n)) =
 \lim_{n\to\infty}
\Big(\mathcal{J}(u_n(\cdot+y_n)) -
\frac{1}{2}\mathcal{J}'(u_n(\cdot+y_n))u_n(\cdot+y_n)\Big) 
\geq \J(u).$$
Since $u\in\cP=X\setminus\tX$, then by Theorem \ref{ThAbstract}
$$c=\inf_{\cN}\J=\J(u).$$
In view of Theorem \ref{ThExp},
$u$ is continuous and decays exponentially.
\end{altproof}

\begin{altproof}{\eqref{eq:minmax}}
Let $u\in X^+\setminus\{0\}$. Observe that the following map
$$\R^+\times\tX\ni (t,\tu)\mapsto -\J(tu^++\tu)\in\R$$
is weakly lower semicontinuous and coercive. Thus there is $t\geq 0$ and $\tu\in \tX$ such that 
$$\J(tu^++\tu)=\sup_{\R^+u\oplus\tX}\J.$$
Condition (A3) implies that $t>0$ and thus $tu^++\tu\in\cN$. Moreover $\J(tu^++\tu)\geq\inf_{\cN}\J$ and
\begin{equation}\label{eq:proofminmax}
\inf_{u\in X^+\setminus\{0\}}\sup_{\R^+u\oplus\tX}\J\geq\inf_{\cN}\J. 
\end{equation}
Since $\inf_{\cN}\J$ is attained by a ground state, we get the equality in \eqref{eq:proofminmax}.
\end{altproof}

\section{Continuity and decay of solutions}
\label{sect:DecayOfSolutions}

We present the results of this section in a general setting. Consider the system of nonlinear Schr\"odinger equations \eqref{eq} under the following assumptions:
\begin{itemize}
\item[(V1)] For any $i=1,2,...,K$, $V_i\in L^{\infty}_{loc}(\R^N)$, 
and
there is $V_0\in\R$ such that 
$$V_i(x)\geq-V_0\textnormal{ for any }x\in\R^N\textnormal{ and }i=1,2,...,K.$$
\item[(A1)] $f_i:\R^N\times\R^K\to \R$ is measurable, $\Z^N$-periodic in $x\in\R^N$ and continuous in $u\in\R^K$ for a.e. $x\in\R^N$. Moreover
there are $a>0$ and $2<p<2^*=\frac{2N}{(N-2)_+}$ such that
$$|f_i(x,u)|\leq a(1+|u|^{p-2})|u_i|\hbox{ for all }u \in\R^K,\; x\in\R^N
\textnormal{ and }i=1,2,...,K.$$
\end{itemize}

Similarly as Pankov \cite{PankovDecay} in case $K=1$, we say that $u=(u_1,u_2,...,u_K):\R^N\to\R^K$ is a {\em weak solution} to system (\ref{eq}) if for any $i=1,2,...,K$
$$u_i\in E_i:=\{v\in H^1(\R^N)|\; (V_i(x)+V_0+1)v\in L^2(\R^N)\}$$
and
$$\int_{\R^N}\nabla u_i\cdot \nabla\varphi +V_i(x)u_i\varphi-f_i(x,u)\varphi \,dx=0\textnormal{ for any }\varphi\in C^{\infty}_0(\R^N).$$

Observe that conditions (V) and (F1)-(F3) imply (V1) and (A1), $E_i=H^1(\R^N)$ and weak solutions coincide with the critical points of $\J$.

\begin{Lem}\label{PropPankov}
Suppose that (V1), (A1) hold and let $u=(u_1,u_2,..,u_K):\R^N\to\R^K$ be a weak solution to 
\eqref{eq}.
Then $u$ is continuous and
$\lim_{x\to\infty}u(x)=0.$
\end{Lem}
\begin{proof}
Let us assume that $N\geq 3$.
Arguing similarly as in the proof of \cite{PankovDecay}[Lemma 1], for $i=1,2,...,K$, $u_i\in E_i\subset L^{2^*}(\R^N)$, operator $H_i=(-\Delta+V_i(x)+V_0+1)$ is positive definite and satisfies the assumptions of the Sobolev estimate theorem \cite{Simon}[Theorem B.2.1].
Since $V_i(x)+V_0+1\in K_N^{loc}$, $V_i(x)+V_0+1>0$ and 
$0<\inf \sigma(H_i)$,  then due to \cite{Simon}[Theorem B.2.1], $H_i^{-1}$ is bounded from $L^s(\R^N)$ to $L^q(\R^N)$ provided that
$$\frac{1}{s}-\frac{1}{q}<\frac{2}{N}$$
and $s\leq q$.
Let $u\in (L^r(\R^N))^K$, where $r\geq 2^*$,
$$A:=\{x\in\R^N|\; |u(x)|\geq 1\}.$$
Note that $A$ has finite measure and
$$h_i^0(x):=\chi_{\R^N\setminus A}(x)((V_0+1)u_i(x)+f_i(x,u(x)))
\in L^{\infty}(\R^N)$$
$$h_i^1(x):=\chi_{A}(x)((V_0+1)u_i(x)+f_i(x,u(x)))
\in L^{s}(\R^N),$$
where $s=\frac{r}{p-1}$.
Hence 
$$u_i=H^{-1}_i(h_i^0+h_i^1)=
H^{-1}_i(h_i^0)+H^{-1}_i(h_i^1)\in L^{\infty}(\R^N)+L^{q}(\R^N)$$
for any
$i=1,2,...,K$.
Observe that if $\frac{1}{s}<\frac{2}{N}$, then we can take $q=\infty$ and $u_i\in L^{\infty}(\R^N)$, otherwise 
$$q<r\frac{1}{(p-1)-\frac{2r}{N}}$$
and we can take any $q=r\frac{1}{1-\eps}$ with $0<\eps<2^*-p$. Starting with $r=2^*$ we continue the above procedure and we obtain
$u_i\in L^{\infty}(\R^N)+L^{s}(\R^N)$, where $s=2^*\big(\frac{1}{1-\eps}\big)^l$ for $l\geq 1$. For sufficiently large $l$,
$\frac{1}{s}<\frac{2}{N}$ and we may take $q=\infty$. Thus 
$u_i\in L^{\infty}(\R^N)$ for any $i=1,2,...,K$.
Now let us define
\begin{equation}\label{DefOfW}
W_i(x)=\left\{
\begin{array}{ll}
    -\frac{f_i(x,u(x))}{u_i(x)},
    &
    \hbox{ if } u_i(x)\neq 0,\\
    0,
    &
    \hbox{ if } u_i(x)= 0
\end{array}
\right.
\end{equation}
for $i=1,2,...,K$.
Since $u_i\in L^{\infty}(\R^N)$, then $W_i\in L^{\infty}(\R^N)$ and
$(-\Delta +V_i+W_i)u_i=0$. In view of \cite{Simon}[Theorem C.1.1] we get the continuity of $u_i$ and by \cite{Simon}[Theorem C.3.1], $u_i(x)\to 0$ as $x\to\infty$ for $i=1,2,...,K$. If $N=1$, then any function from $H^1(\R)$ is continuous and decays at infinity. If $N=2$, then due to Sobolev embedding theorems we obtain that $u_i\in L^{\infty}(\R^N)$, and similarly as above we show that $u_i$ is continuous and decays at infinity.
\end{proof}

Now let us consider the following assumptions:
\begin{itemize}
\item[(V2)] For any $i=1,2,...,K$, the essential spectrum $\sigma_{ess}(-\Delta+V_i)$ of the operator $-\Delta+V_i$ in $L^2(\R^N)$ does not contain the point $0$.
\item[(A2)] For any $i=1,2,...,K$
$$\lim_{|u|\to 0} \essup\frac{|f_i(x,u)|}{|u_i|}=0.$$
\end{itemize}

If $V_i\in L^{\infty}(\R^N)$ is periodic, then by \cite{ReedSimon} we have
$$\sigma_{ess}(-\Delta+V_i)=\sigma(-\Delta+V_i)$$
and (V) implies (V2) and (F3) implies (A2).

\begin{Th}\label{ThExp}
Suppose that (V1), (V2), (A1), (A2) hold and $u=(u_1,u_2,...,u_K)$ is a weak solution to \eqref{eq}.
Then $u$ is continuous and there exist $\alpha>0$ and $C>0$ such that
\begin{equation}\label{estimateexp}
|u(x)|\leq C \exp(-\alpha|x|).
\end{equation}
\end{Th}
\begin{proof}
Let us fix $1\leq i\leq K$. Consider $W_i$ given by \eqref{DefOfW}
and in view of Lemma \ref{PropPankov} and (A2) we get 
$$\lim_{|x|\to\infty} \essup W_i(x)=0.$$
Since $W_i\in L^{\infty}(\R^N)$, then
$W_i$ is relatively compact with respect to $-\Delta$ (see section XIII.4 in \cite{ReedSimon}). Thus $W_i(-\Delta +1)^{-1}$ is a compact operator.
Note that $V_i$ is bounded from below and there is $\lambda>1$ such that
$-\lambda\notin\sigma(-\Delta+V_i)$.
Then $(-\Delta+V_i+\lambda)^{-1}$ is bounded and
$$(-\Delta + V_i+ \lambda)^{-1}=(-\Delta + 1)^{-1}-(-\Delta + 1)^{-1}(V_i+\lambda-1)(-\Delta + V_i+ \lambda)^{-1},$$
hence $W_i(-\Delta+V_i+\lambda)^{-1}$ is compact.
Therefore $W_i$ is relatively compact with respect to $-\Delta + V_i$ (see section XIII.4 in \cite{ReedSimon}) and
$$\sigma_{ess}(-\Delta + V_i)=\sigma_{ess}(-\Delta + V_i+W_i).$$
Since $(-\Delta + V_i+W_i)u_i=0$, then $0$ is an isolated eigenvalue of finite multiplicity of the operator 
$-\Delta + V_i+W_i$ and in view of \cite{Simon}[Theorem C.3.4], the eigenfunction $u_i$ satisfies 
$$|u_i(x)|\leq C_i \exp(-\alpha|x|)$$
for some $\alpha>0$ and $C_i>0$. Thus \eqref{estimateexp} holds.
\end{proof}

\begin{Rem}
In Theorem \ref{ThExp}, if $\Sigma=\inf_{i=1,2,...,K}\inf \sigma_{ess}(-\Delta+V_i)>0$, then
due to \cite{Simon}[Theorem C.3.5] for any $\alpha<\sqrt{2\Sigma}$ there is $C>0$ such that \eqref{estimateexp} holds.
\end{Rem}

{\bf Acknowledgements.} The author was partially supported by National Science Centre, Poland (Grant No. 2014/15/D/ST1/03638) and he would like to thank the referee for many valuable comments. The author is also indebted to Andrzej Szulkin and Wojciech Kryszewski for their remarks helping to improve the paper.

\parbox{9cm}{
\noindent \\{\sc Address of the author:}\\
Jaros\l aw Mederski\\
 Nicolaus Copernicus University \\
 Faculty of Mathematics and Computer Science\\
 ul. Chopina 12/18, 87-100 Toru\'n, Poland\\
 jmederski@mat.umk.pl\\
 }
 
\end{document}